\documentclass[reqno,a4paper]{amsart}
\usepackage[utf8]{inputenc}
\usepackage{amsmath,amssymb}
\usepackage{mathrsfs}
\usepackage{yhmath}
\usepackage{booktabs,url}

\newcommand{\de}[0]{\mathrel{\mathop:}=}
\newcommand{\ie}[0]{\mathrm{i}}
\newcommand{\dif}[1]{\mathrm{d}#1}

\newcommand{\R}[0]{\mathbb{R}}

\newcommand{\N}[0]{\mathbb{N}}

\usepackage{amsmath}

\newtheorem{theorem}{Theorem}

\newtheorem{corollary}{Corollary}
\newtheorem{lemma}{Lemma}

\title[On explicit estimates for $S(t)$, $S_1(t)$, and $\zeta\left(1/2+\mathrm{i}t\right)$ under RH]{On explicit estimates for $S(t)$, $S_1(t)$, and $\zeta\left(1/2+\mathrm{i}t\right)$ under the Riemann Hypothesis}
\author{Aleksander Simoni\v{c}}
\address{School of Science, The University of New South Wales (Canberra), ACT, Australia}
\email{a.simonic@student.adfa.edu.au}

\begin{document}

\begin{abstract}
Assuming the Riemann Hypothesis, we provide explicit upper bounds for moduli of $S(t)$, $S_1(t)$, and $\zeta\left(1/2+\mathrm{i}t\right)$ while comparing them with recently proven unconditional ones. As a corollary we obtain a conditional explicit bound on gaps between consecutive zeros of the Riemann zeta-function.
\end{abstract}

\maketitle
\thispagestyle{empty}

\section{Introduction}

Let $\zeta(s)$ be the Riemann zeta-function and $s=\sigma+\ie t$, where $\sigma$ and $t$ are real numbers. For $t\in\R$ distinct from the ordinate of any nontrivial zero $\rho=\beta+\ie\gamma$ of $\zeta(s)$ let
\begin{equation*}
\label{eq:Sfunction}
S(t)=\frac{1}{\pi}\arg{\zeta\left(\frac{1}{2}+\ie t\right)},
\end{equation*}
and
\[
S(t)=\lim_{\varepsilon\downarrow0}\frac{1}{2}\left(S(t+\varepsilon)+S(t-\varepsilon)\right)
\]
otherwise, where $\arg{\zeta\left(1/2+\ie t\right)}$ is defined by the continuous extension of $\arg{\zeta(s)}$ along straight lines connecting $s=2$, $s=2+\ie t$, $s=1/2+\ie t$, and oriented in this direction, where $\arg{\zeta(2)}=0$. The importance of studying the behaviour of $S(t)$ lies within the Riemann--von Mangoldt formula
\begin{equation}
\label{eq:RvM}
N(T) = \frac{T}{2\pi}\log{\frac{T}{2\pi e}} + \frac{7}{8} + S(T) + R(T),
\end{equation}
where $N(T)$ is the number of $\rho$ with $0<\gamma\leq T$, and $R(T)=O\left(1/T\right)$. It is well known that
\begin{equation}
\label{eq:Sclassic}
S(t)=O\left(\log{t}\right),
\end{equation}
and the most recent explicit bound for~\eqref{eq:Sclassic} is
\begin{equation}
\label{eq:boundS}
\left|S(t)\right| \leq 0.11\log{t} + 0.29\log{\log{t}} + 2.29
\end{equation}
for $t\geq e$, see~\cite[Corollary 1]{Platt}. At present, there is no unconditional improvement on~\eqref{eq:Sclassic}. However, Cram\'{e}r~\cite{CramerLH} proved that the Lindel\"{o}f Hypothesis (LH),
\begin{equation}
\label{eq:LH}
\zeta\left(\frac{1}{2}+\ie t\right) \ll_{\varepsilon} |t|^{\varepsilon}
\end{equation}
for every $\varepsilon>0$, guarantees $S(t)=o\left(\log{t}\right)$, and Littlewood proved in~\cite[Theorem 11]{LittlewoodOnTheZeros} that the Riemann Hypothesis (RH), namely that in all $\rho$ we have $\beta=1/2$, provides a quantitative version of this estimate since then
\begin{equation}
\label{eq:SRH}
\left|S(t)\right| \leq \left(C_0+o(1)\right)\frac{\log{t}}{\log{\log{t}}}
\end{equation}
for some $C_0>0$ and $t$ large. Selberg contributed in~\cite[Theorem 1]{SelbergOnTheRemainder} a different proof of~\eqref{eq:SRH} via his unconditional equation~\eqref{eq:selberg}. The order of magnitude of~\eqref{eq:SRH} has never been improved, and the efforts have been concentrated in optimizing the value of $C_0$. By means of Selberg's method, Ramachandra and Sankaranarayanan~\cite{Ramachandra} obtained $C_0=1.12$, and Fujii~\cite{FujiiAnExplicit} improved\footnote{The value of $C_0$ is incorrectly calculated there to $0.67$, but corrected and even slightly improved to $C_0=0.81$ in~\cite[pp.~146--147]{FujiiANote}.} this to $C_0=0.83$. Goldston and Gonek~\cite{GoldstonGonek} discovered a different method based on a variation of the Guinand--Weil exact formula~\cite[Equation 25.10]{IKANT}. They obtained $C_0=1/2$, which was later improved by Carneiro, Chandee, and Milinovich~\cite[Theorem 2]{CCMBounding} to $C_0=1/4$.

The antiderivative of $S(t)$,
\[
S_1(t) \de \int_{0}^{t} S(u)\dif{u}
\]
for $t\geq0$, has been also intensively studied. Littlewood proved in the same paper~\cite{LittlewoodOnTheZeros} that
\begin{equation*}
S_1(t)=O\left(\log{t}\right),
\end{equation*}
that LH implies\footnote{As observed by Ghosh and Goldston, LH is in fact equivalent to $S_1(t)=o\left(\log{t}\right)$, see~\cite[p.~335]{Titchmarsh} for a sketch of the proof.} $S_1(t)=o\left(\log{t}\right)$, and that RH refines this to
\begin{equation}
\label{eq:S1RH}
\left|S_1(t)\right| \leq \left(C_1+o(1)\right)\frac{\log{t}}{\left(\log{\log{t}}\right)^2}
\end{equation}
for some $C_1>0$ and $t$ large. We\footnote{This follows from~\cite[Theorem 2.2]{TrudgianImpr2}, and the fact that $\left|S_1(168\pi)\right|\leq0.987$ and \eqref{eq:boundS1} is true for $1\leq t\leq168\pi$.} also know that
\begin{equation}
\label{eq:boundS1}
\left|S_1(t)\right| \leq 0.059\log{t} + 3.054
\end{equation}
for $t\geq1$. As for $S(t)$, the order of magnitude of~\eqref{eq:S1RH} has never been improved and several authors have given explicit values for $C_1$. Following Selberg's method, Karatsuba and Korol\"{e}v~\cite[Theorem 2 of Chapter III]{KaratsubaKorolevArgument} obtained $C_1=40$, and Fujii~\cite{FujiiANote} improved this to $C_1=0.51$. Goldston--Gonek's approach was used in~\cite[Theorem 1]{CCMBounding} to get $C_1=\pi/24$.

Currently the best known result concerning~\eqref{eq:LH} is
\[
\zeta\left(\frac{1}{2}+\ie t\right) \ll_{\varepsilon} |t|^{\frac{13}{84}+\varepsilon}
\]
for $\varepsilon>0$, due to Bourgain~\cite{BourgainDecoupling}. Because RH implies LH, see~\cite[Theorem 14.2]{Titchmarsh}, it should be expected that something more precise than~\eqref{eq:LH} is possible to state on the truth of RH. As a corollary to his methods for $S(t)$ and related functions, Littlewood~\cite[Theorem 12]{LittlewoodOnTheZeros} proved that RH implies
\begin{equation}
\label{eq:zetaRH}
\left|\zeta\left(\frac{1}{2}+\ie t\right)\right| \leq \exp{\left(\left(C_{00}+o(1)\right)\frac{\log{t}}{\log{\log{t}}}\right)}
\end{equation}
for some $C_{00}>0$ and $t$ large. Also in this case the order of magnitude of~\eqref{eq:zetaRH} has never been improved and efforts were put into obtaining explicit values for $C_{00}$. It was proved in~\cite{Ramachandra} by Selberg's method that $C_{00}=0.47$, while Soundararajan~\cite{SoundMoments} used its refined version to obtain $C_{00}=0.373$. Goldston--Gonek's method was used by Chandee and Soundararajan~\cite{ChandeeSound}, and also in~\cite[Corollary 4]{Carneiro}, to give $C_{00}=\log{\sqrt{2}}$. The classical subconvexity bound was made explicit by Hiary~\cite{HiaryAnExplicit}, who proved\footnote{This result relies on the erroneous Cheng--Graham explicit version of van der Corput's second derivative test, see \cite{Patel} for details. After accounting for this correction, the constant from \eqref{eq:hiary} changes to $0.77$. Our results are not affected by this.} that
\begin{equation}
\label{eq:hiary}
\left|\zeta\left(\frac{1}{2}+\ie t\right)\right| \leq 0.63 t^{\frac{1}{6}}\log{t}
\end{equation}
holds for $t\geq 3$.

None of the previously mentioned authors obtained explicit $o$-terms in~\eqref{eq:SRH},~\eqref{eq:S1RH}, and~\eqref{eq:zetaRH}. We are not able to deduce from these estimates any concrete values for $t$ when RH provides better bounds compared to unconditional and explicit ones~\eqref{eq:boundS},~\eqref{eq:boundS1}, and~\eqref{eq:hiary}. Additionally, after choosing a method, we are interested in optimization of certain parameters in order to obtain explicit conditional bounds which become better than unconditional ones at the lowest $t$ possible. The main result of this paper is the following theorem.

\begin{theorem}
\label{thm:main}
Assume the Riemann Hypothesis. Let
\begin{equation}
\label{eq:M}
\mathcal{M}\left(a,b,c;t\right)\de a + \frac{b}{\left(\log{t}\right)^{c}\log{\log{t}}}
\end{equation}
for positive real numbers $a$, $b$ and $c$. Then we have
\[
\left|S(t)\right| \leq \mathcal{M}\left(0.759282,20.1911,0.285;t\right)\frac{\log{t}}{\log{\log{t}}}
\]
for $t\geq10^{2465}$,
\[
\left|S_1(t)\right| \leq \mathcal{M}\left(0.653,60.12,0.2705;t\right)\frac{\log{t}}{\left(\log{\log{t}}\right)^2}
\]
for $t\geq 10^{208}$, and
\[
\left|\zeta\left(\frac{1}{2}+\ie t\right)\right| \leq \exp{\left(\mathcal{M}\left(0.5,6.992,0.252;t\right)
\frac{\log{t}}{\log{\log{t}}}\right)}
\]
for $t\geq10^{40}$.
\end{theorem}

In particular, Theorem \ref{thm:main}, together with \eqref{eq:boundS},\footnote{We are also using \cite[Corollary 1]{BrentPlattTrudgianPNT} that $\left|S(t)\right|\leq0.28\log{t}$ for $t\geq2\pi$. Note that this estimate is for $2\pi\leq t\leq10^8$ better than \eqref{eq:boundS}.} \eqref{eq:boundS1}, and \eqref{eq:hiary}, asserts the following.

\begin{corollary}
Assume the Riemann Hypothesis. Then we have
\[
\left|S(t)\right| \leq \frac{0.96\log{t}}{\log{\log{t}}}, \quad \left|S_1(t)\right| \leq \frac{2.491\log{t}}{\left(\log{\log{t}}\right)^2}, \quad
\left|\zeta\left(\frac{1}{2}+\ie t\right)\right| \leq \exp{\left(\frac{0.995\log{t}}{\log{\log{t}}}\right)}
\]
for $t\geq2\pi$.
\end{corollary}

These bounds are better than unconditional ones for $t\geq10^{2510}$, $t\geq10^{208.4}$ and $t\geq10^{72}$, respectively.

Our proof of Theorem~\ref{thm:main} for $S(t)$ and $S_1(t)$ follows Fujii's elaborations on Selberg's method. However, we are working with the parameterized function $\sigma_{x,\alpha}$, see Lemma~\ref{lem:selberg} for details, which in turn improves Fujii's values for $C_0$ and $C_1$; we obtain $C_0=0.543$ and $C_1=0.337$ with fully explicit $o$-terms. Although this small modification still produced weaker constants compared to the best known ones, the method is simple enough to handle error terms with ease. It may be interesting to see if even further refinements of this classical method are possible as well as to investigate potential improvements of unconditional results in~\cite{KaratsubaKorolevArgument}. Proof for the modulus of $\zeta$ on the critical line follows~\cite[pp.~985--987]{SoundMoments}, therefore we obtain $C_{00}=0.373$ with a fully explicit $o$-term.

It should be noted that there exists also an iterative generalization of $S_1(t)$, namely
\[
S_n(t) \de \int_{0}^{t}S_{n-1}(u)\dif{u} + \delta_n
\]
for $t>0$ and $n\geq1$, where $S_0(u)\de S(u)$ and $\delta_n$ are certain constants, see, e.g.,~\cite{Carneiro}. It is known that RH implies
\begin{equation}
\label{eq:SnRH}
\left|S_n(t)\right| \leq \left(C_n+o(1)\right)\frac{\log{t}}{\left(\log{\log{t}}\right)^{n+1}}
\end{equation}
for $C_n>0$, $t$ large and implied constants depend on $n$, a result also due to Littlewood~\cite[Theorem 11]{LittlewoodOnTheZeros}. It is also known that RH is equivalent to the assertion that $S_n(t)=o\left(t^{n-2}\right)$ for every $n\geq3$, see~\cite[Theorem 4]{FujiiOnTheZeros}. Explicit values for $C_n$ were firstly obtained by Wakasa~\cite{Wakasa} by following Fujii's work, and significantly improved in~\cite{Carneiro}. Although there is no serious obstacle to derive~\eqref{eq:SnRH} in fully explicit form with constants $C_n$ not worse than those in~\cite{Wakasa}, we are not considering this in the present paper.

One immediate application of Theorem~\ref{thm:main} is linked to gaps between ordinates of consecutive nontrivial zeros. It follows from~\eqref{eq:RvM} and~\eqref{eq:Sclassic} that such gaps are bounded, with the first gap being the largest, see~\cite[Lemma 3]{Sim.Lehmer}. Littlewood~\cite{LittlewoodTwoNotes} proved that
\begin{equation}
\label{eq:LittlewoodGaps}
\gamma'-\gamma \leq \frac{32}{\log{\log{\log{\gamma}}}}
\end{equation}
for the ordinates of two consecutive nontrivial zeros $\gamma'\geq\gamma\geq T$ with $T$ large, see also~\cite[pp.~224--227]{Titchmarsh}. The constant in~\eqref{eq:LittlewoodGaps} was improved to $\pi/2+o(1)$ in~\cite[Theorem 1]{HallHayman}. On RH we have $\gamma'-\gamma\ll 1/\log{\log{\gamma}}$ by~\eqref{eq:RvM} and~\eqref{eq:SRH}, and~\cite[Corollary 1]{GoldstonGonek} asserts further that $\gamma'-\gamma\leq \left(\pi+o(1)\right)/\log{\log{\gamma}}$. Theorem~\ref{thm:main} enables us to state a similar result in a completely explicit form.

\begin{corollary}
\label{cor:gaps}
Assume the Riemann Hypothesis. Then
\[
\gamma'-\gamma \leq \mathcal{M}\left(9.55,253.82,0.285;\gamma\right)\frac{1}{\log{\log{\gamma}}},
\]
where $\gamma'\geq\gamma\geq10^{2465}$ are the ordinates of two consecutive nontrivial zeros of the Riemann zeta-function, and $\mathcal{M}$ is defined by~\eqref{eq:M}. In particular, $\gamma'-\gamma \leq 12.05/\log{\log{\gamma}}$ for $\gamma'\geq\gamma\geq10^{2465}$.
\end{corollary}

The outline of this paper is as follows. In Section~\ref{sec:boundsI} we use the formulae of Selberg and Soundararajan to derive explicit approximations of $S(t)$, $S_1(t)$, and $\log{\left|\zeta\left(1/2+\ie t\right)\right|}$ by segments of Dirichlet series, see Theorems~\ref{thm:selberg},~\ref{thm:selberg2}, and~\ref{thm:logzeta}. In Section~\ref{sec:boundsII} we estimate certain sums over prime numbers in order to provide upper bounds for $\left|S(t)\right|$, $\left|S_1(t)\right|$, and $\left|\zeta\left(1/2+\ie t\right)\right|$, see Corollaries~\ref{cor:S},~\ref{cor:S1}, and~\ref{cor:sound}. In Section~\ref{sec:proof} we then use results from Section~\ref{sec:boundsII} to obtain the proofs of Theorem~\ref{thm:main} and Corollary~\ref{cor:gaps}.

\section{Derivation of general bounds I}
\label{sec:boundsI}

Useful representations of $S(t)$ and $S_1(t)$ are through integration, namely
\begin{equation}
\label{eq:intS}
\pi S(t) = -\int_{\frac{1}{2}}^{\infty}\Im\left\{\frac{\zeta'}{\zeta}(s)\right\}\dif{\sigma}
\end{equation}
and
\begin{equation}
\label{eq:intS1}
\pi S_1(t) = -\mathrm{p.v.}\int_{\frac{1}{2}}^{\infty}\log{\left|\zeta(\sigma)\right|}\dif{\sigma} - \int_{\frac{1}{2}}^{\infty} \left(\sigma-\frac{1}{2}\right)\Re\left\{\frac{\zeta'}{\zeta}\left(s\right)\right\}\dif{\sigma}
\end{equation}
for any nonzero $t$, see Lemmas~4 and 8 of Chapter II in~\cite{KaratsubaKorolevArgument} for detailed proofs. Selberg~\cite[Lemma 2]{SelbergOnTheNormal} discovered a profound connection between the logarithmic derivative of the Riemann zeta-function and special truncated Dirichlet series. For $x\geq2$ and a positive integer $n$ define
\[
\Lambda_x(n) \de \left\{
\begin{array}{ll}
    \Lambda(n), & 1\leq n\leq x, \\
    \Lambda(n)\frac{\log{\left(x^2/n\right)}}{\log{x}}, & x< n\leq x^2,
\end{array}
\right.
\]
where $\Lambda(n)$ is the von Mangoldt function. Observe that $\Lambda_x(n)\leq\Lambda(n)$. Then
\begin{multline}
\label{eq:selberg}
\frac{\zeta'}{\zeta}(s) = -\sum_{n\leq x^2}\frac{\Lambda_x(n)}{n^s} + \frac{x^{2(1-s)}-x^{1-s}}{(1-s)^2\log{x}} \\
+ \frac{1}{\log{x}}\sum_{q=1}^{\infty}\frac{x^{-2q-s}-x^{-2(2q+s)}}{(2q+s)^2} + \frac{1}{\log{x}}\sum_{\rho}\frac{x^{\rho-s}-x^{2\left(\rho-s\right)}}{\left(s-\rho\right)^2}
\end{multline}
for $s\notin\{1\}\cup\left\{-2q\colon q\in\N\right\}$ and $s\neq\rho$, where $\rho$ is a nontrivial zero. For the proof of~\eqref{eq:selberg} see, e.g.,~\cite[Theorem 14.20]{Titchmarsh}. This formula is true without any hypothesis, but treatment of the series through the nontrivial zeros is much simpler if we assume RH.

Equality~\eqref{eq:selberg} is used in the following two sections to obtain bounds on $\left|S(t)\right|$ and $\left|S_1(t)\right|$. It could be also used for estimating the modulus of $\zeta$ on the critical line. However, while studying moments of $\zeta$ on the critical line, Soundararajan discovered in~\cite[Lemma 1]{SoundMoments} an even better approach, namely through a similar and unconditional formula
\begin{multline}
\label{eq:sound}
-\frac{\zeta'}{\zeta}(s) = \sum_{n\leq x}\frac{\Lambda(n)\log{\frac{x}{n}}}{n^{s}\log{x}} + \frac{1}{\log{x}}\left(\frac{\zeta'}{\zeta}(s)\right)'-\frac{x^{1-s}}{(1-s)^2\log{x}} \\
+ \frac{1}{\log{x}}\sum_{q=1}^{\infty}\frac{x^{-2q-s}}{(2q+s)^2} + \frac{1}{\log{x}}\sum_{\rho}\frac{x^{\rho-s}}{\left(\rho-s\right)^2},
\end{multline}
valid for $s\notin\{1\}\cup\left\{-2q\colon q\in\N\right\}$ and $s\neq\rho$. In Section~\ref{sec:zetaCL} we use~\eqref{eq:sound} instead of~\eqref{eq:selberg} in order to obtain an explicit estimate~\eqref{eq:zetaRH} with $C_{00}=0.373$.

It is well known that $\zeta'(s)/\zeta(s)-\sum_{\rho}\left(s-\rho\right)^{-1}=O\left(\log{t}\right)$, which is a consequence of the functional equation and the Weierstrass product representation of $\xi(s)$, and the Stirling formula for $\Gamma'(s)/\Gamma(s)$. The next lemma is an explicit version of this result and is crucial for bounding sums over $\rho$ in~\eqref{eq:selberg} and~\eqref{eq:sound}.

\begin{lemma}
\label{lem:logder}
Let $s=\sigma+\ie t$ and $s_0=\sigma_0+\ie t$. For $|\sigma|\leq 2$, $\left|\sigma_0\right|\leq 2$, $|t|\geq10$, $s\neq\rho$ and $s_0\neq\rho$ we have
\[
\left|\frac{\zeta'}{\zeta}(s) - \sum_{\rho}\frac{1}{s-\rho}\right| \leq \frac{1}{2}\log{|t|} + 3
\]
and
\[
\left|\frac{\zeta'}{\zeta}(s) - \frac{\zeta'}{\zeta}\left(s_0\right) - \sum_{\rho}\left(\frac{1}{s-\rho}-\frac{1}{s_0-\rho}\right)\right| \leq
\frac{\left|\sigma-\sigma_0\right|}{|t|}\left(\frac{\pi}{4}+\frac{2}{|t|}\right).
\]
\end{lemma}

\begin{proof}
The first inequality is~\cite[Lemma 2 on p.~441]{KaratsubaKorolevArgument}. By the course of their proof we have
\begin{equation}
\label{eq:KKlogder}
\frac{\zeta'}{\zeta}(s) - \sum_{\rho}\frac{1}{s-\rho} = \sum_{n=1}^{\infty} \left(\frac{1}{s+2n}-\frac{1}{2n}\right) + \frac{\gamma_0+\log{\pi}}{2} + \frac{1}{1-s},
\end{equation}
where $\gamma_0$ is the Euler--Mascheroni constant. Because
\[
\left|s+2n\right| = \sqrt{\left(\sigma+2n\right)^2+t^2} \geq \sqrt{\left(2n-2\right)^2+t^2}
\]
with the same inequality also for $s_0$ in place of $s$, we obtain
\begin{flalign*}
\left|\sum_{n=1}^{\infty} \left(\frac{1}{s+2n}-\frac{1}{s_0+2n}\right)\right| &\leq \left|\sigma-\sigma_0\right|\sum_{n=1}^{\infty}\frac{1}{\left|s+2n\right|\cdot\left|s_0+2n\right|} \\
&\leq \frac{\left|\sigma-\sigma_0\right|}{t^2}\sum_{n=1}^{\infty}\frac{1}{1+\left(\frac{2n-2}{|t|}\right)^2} \\
&\leq \frac{\left|\sigma-\sigma_0\right|}{t^2}\left(1+\int_{1}^{\infty}\frac{\dif{x}}{1+\left(\frac{2x-2}{|t|}\right)^2}\right) \\
&\leq \frac{\left|\sigma-\sigma_0\right|}{t^2}\left(1+\frac{\pi|t|}{4}\right).
\end{flalign*}
Also,
\[
\left|\frac{1}{1-s}-\frac{1}{1-s_0}\right| = \frac{\left|\sigma-\sigma_0\right|}{\left|1-s\right|\cdot\left|1-s_0\right|} \leq \frac{\left|\sigma-\sigma_0\right|}{t^2}.
\]
The second inequality now follows from~\eqref{eq:KKlogder}.
\end{proof}

The following lemma gives an explicit form of the ``remainder term'' of~\eqref{eq:selberg} under RH, and should be compared with~\cite[Lemma 2]{FujiiANote}. It is used in Sections~\ref{sec:S} and~\ref{sec:S1}.

\begin{lemma}
\label{lem:selberg}
Assume the Riemann Hypothesis. For $\alpha>0$ define
\[
\sigma_{x,\alpha}\de \frac{1}{2} + \frac{\alpha}{\log{x}}, \quad r(x,t,\alpha) \de \sum_{n\leq x^2}\frac{\Lambda_x(n)}{n^{\sigma_{x,\alpha}+\ie t}}.
\]
Let $1\leq\alpha\leq2$, $t\geq10$, $e^{4/3}\leq x_0\leq x\leq t^2$, and $\sigma\geq\sigma_{x,\alpha}$. Then we have
\begin{flalign}
\left|\frac{\zeta'}{\zeta}(s)+\sum_{n\leq x^2}\frac{\Lambda_x(n)}{n^s}\right| &\leq
\frac{e^{\alpha}\left(1+x^{\frac{1}{2}-\sigma}\right)x^{\frac{1}{2}-\sigma}}{\alpha e^{\alpha}-1-e^{-\alpha}}\left(\left|r(x,t,\alpha)\right|+\frac{1}{2}\log{t}\right) \nonumber \\
&+\frac{\alpha e^{\alpha}A\left(x_0\right)+3\left(1+e^{\alpha}\right)}{\alpha e^{\alpha}-1-e^{-\alpha}}
x^{\frac{1}{2}-\sigma}, \label{eq:main2}
\end{flalign}
where
\begin{equation}
\label{eq:A}
A\left(x_0\right)\de \frac{1}{\log{x_0}}\left(1+\frac{1}{\sqrt{x_0}} + \frac{\left(x_0^{2}+1\right)\sqrt{x_0}+1}{x_0^2\left(x_0^4-1\right)}\right).
\end{equation}
\end{lemma}

\begin{proof}
Let $s_{x,\alpha}=\sigma_{x,\alpha}+\ie t$. We are going to use Selberg's formula~\eqref{eq:selberg} with the assumption $\rho=1/2+\ie\gamma$. We have
\[
\frac{1}{\log{x}}\left|\sum_{\rho}\frac{x^{\rho-s}-x^{2\left(\rho-s\right)}}{\left(s-\rho\right)^2}\right| \leq \frac{1}{\alpha}\left(1+x^{\frac{1}{2}-\sigma}\right)x^{\frac{1}{2}-\sigma}F\left(s_{x,\alpha}\right),
\]
where
\begin{equation}
\label{eq:F}
F(s) \de \sum_{\gamma} \frac{\sigma-\frac{1}{2}}{\left(\sigma-\frac{1}{2}\right)^2+\left(t-\gamma\right)^2}.
\end{equation}
We also have
\[
\left|\frac{x^{2(1-s)}-x^{1-s}}{(1-s)^2\log{x}}\right| \leq \frac{1+x^{-\frac{1}{2}}}{\log{x}}x^{\frac{1}{2}-\sigma}
\]
and
\begin{flalign*}
\frac{1}{\log{x}}\left|\sum_{q=1}^{\infty}\frac{x^{-2q-s}-x^{-2(2q+s)}}{(2q+s)^2}\right| &\leq \frac{x^{\frac{1}{2}-\sigma}}{\log{x}}\sum_{q=1}^{\infty} \left(x^{-2q-\frac{3}{2}} + x^{-4q-2}\right) \\
&\leq\frac{x^{\frac{5}{2}}+x^{\frac{1}{2}}+1}{x^2\left(x^4-1\right)\log{x}}x^{\frac{1}{2}-\sigma}
\end{flalign*}
because $\left|1-s\right|^2\geq t^2\geq x$, $\left|2q+s\right|^2\geq t^2\geq x$, $\sigma>1/2$ and $\left|x^{1-2s}\right|\leq 1$. Therefore,
\begin{equation}
\label{eq:main1}
\left|\frac{\zeta'}{\zeta}(s)+\sum_{n\leq x^2}\frac{\Lambda_x(n)}{n^s}\right| \leq \frac{1}{\alpha}\left(1+x^{\frac{1}{2}-\sigma}\right)x^{\frac{1}{2}-\sigma}F\left(s_{x,\alpha}\right) + A\left(x_0\right)x^{\frac{1}{2}-\sigma},
\end{equation}
since $A(x)$, defined by~\eqref{eq:A}, is a decreasing function. We also have
\[
\Re\left\{\frac{1}{s-\rho}\right\} = \frac{\sigma-\frac{1}{2}}{\left(\sigma-\frac{1}{2}\right)^2+\left(t-\gamma\right)^2}.
\]
Observe that $\sigma_{x,\alpha}\leq 2$. By Lemma~\ref{lem:logder} and inequality~\eqref{eq:main1} we obtain
\begin{multline}
\label{eq:intermed}
\left|F\left(s\right) + \Re\left\{\sum_{n\leq x^2}\frac{\Lambda_x(n)}{n^s}\right\}\right| \leq \frac{1}{\alpha}\left(1+x^{\frac{1}{2}-\sigma}\right)x^{\frac{1}{2}-\sigma}F\left(s_{x,\alpha}\right) \\
+ A\left(x_0\right)x^{\frac{1}{2}-\sigma} + \frac{1}{2}\log{t} + 3
\end{multline}
for $\sigma\in\left[\sigma_{x,\alpha},2\right]$. Taking $\sigma=\sigma_{x,\alpha}$ into~\eqref{eq:intermed}, we derive
\begin{equation}
\label{eq:important}
\left(1-\frac{1+e^{-\alpha}}{\alpha e^{\alpha}}\right)F\left(s_{x,\alpha}\right) \leq
\left|r(x,t,\alpha)\right|+ e^{-\alpha}A\left(x_0\right) + \frac{1}{2}\log{t} + 3.
\end{equation}
Because
\[
1-\frac{1+e^{-\alpha}}{\alpha e^{\alpha}} \geq 1-\frac{1+e^{-1}}{e} > 0,
\]
taking~\eqref{eq:important} into~\eqref{eq:main1} gives the bound from the lemma.
\end{proof}

\subsection{Approximation of $S(t)$ by segments of Dirichlet series}
\label{sec:S}

Using Lemma~\ref{lem:selberg}, we are going to prove an explicit version of the approximation of $S(t)$ by segments of Dirichlet series, see also~\cite[Theorem 14.21]{Titchmarsh}.

\begin{theorem}
\label{thm:selberg}
Assume the Riemann Hypothesis. Let $1\leq\alpha\leq2$, $t\geq t_0\geq10$, and $e^{4/3}\leq x_0\leq x\leq t^2$. With the notations from Lemma~\ref{lem:selberg} we then have
\begin{multline*}
\left|S(t)+\frac{1}{\pi}\sum_{n\leq x^2}\frac{\Lambda_x(n)\sin{\left(t\log{n}\right)}}{n^{\sigma_{x,\alpha}}\log{n}}\right| \\
\leq a_1(\alpha)\frac{|r(x,t,\alpha)|}{\log{x}} + b_1(\alpha)\frac{\log{t}}{\log{x}} + \frac{c_1\left(x_0,\alpha\right)}{\log{x}} + \frac{d_1\left(t_0,\alpha\right)}{\log^{2}{x}},
\end{multline*}
where
\begin{gather*}
a_1(\alpha) \de \frac{\left(\pi+2\right)\alpha^2 e^{\alpha}+2\left(1+\frac{1}{2}e^{-\alpha}\right)}{2\pi\left(\alpha e^{\alpha}-1-e^{-\alpha}\right)},\;
b_1(\alpha) \de \frac{\pi\alpha^2 e^{\alpha}+2(1+\alpha)+(1+2\alpha)e^{-\alpha}}{4\pi\left(\alpha e^{\alpha}-1-e^{-\alpha}\right)}, \\
c_1\left(x_0,\alpha\right) \de \frac{\alpha\left(2+(2+\pi)\alpha\right)A\left(x_0\right)+6(1+\alpha)\left(1+e^{-\alpha}\right)+3\pi\alpha^{2}e^{\alpha}}{2\pi\left(\alpha e^{\alpha}-1-e^{-\alpha}\right)}, \\
d_1\left(t_0,\alpha\right) \de \frac{\alpha^2}{\pi t_0}\left(\frac{\pi}{4}+\frac{2}{t_0}\right),
\end{gather*}
and $A\left(x_0\right)$ is defined by~\eqref{eq:A}.
\end{theorem}

\begin{proof}
Let $s_{x,\alpha}=\sigma_{x,\alpha}+\ie t$. By~\eqref{eq:intS} we can write $\pi S(t) =  J_1 + J_2 + J_3$, where
\begin{gather*}
J_1 \de -\int_{\sigma_{x,\alpha}}^{\infty}\Im\left\{\frac{\zeta'}{\zeta}\left(s\right)\right\}\dif{\sigma}, \quad J_2 \de -\left(\sigma_{x,\alpha}-\frac{1}{2}\right)\Im\left\{\frac{\zeta'}{\zeta}\left(s_{x,\alpha}\right)\right\}, \\
J_3 \de \int_{\frac{1}{2}}^{\sigma_{x,\alpha}}\Im\left\{\frac{\zeta'}{\zeta}\left(s_{x,\alpha}\right)-\frac{\zeta'}{\zeta}\left(s\right)\right\}\dif{\sigma}.
\end{gather*}
This implies
\begin{equation}
\label{eq:S}
\left|S(t)-\frac{J_1}{\pi}\right| \leq \frac{\alpha}{\pi\log{x}}\left|\frac{\zeta'}{\zeta}\left(s_{x,\alpha}\right)\right|+\frac{\left|J_3\right|}{\pi}.
\end{equation}
The main idea is to use Lemma~\ref{lem:selberg} in combination with Lemma~\ref{lem:logder} in order to estimate the right hand side of~\eqref{eq:S}.

By Lemma~\ref{lem:selberg} for $s=s_{x,\alpha}$ we have
\begin{equation}
\label{eq:J2}
\left|\frac{\zeta'}{\zeta}\left(s_{x,\alpha}\right)\right| \leq \frac{\alpha e^{\alpha}\left|r(x,t,\alpha)\right|+\frac{1+e^{-\alpha}}{2}\log{t}+\alpha A\left(x_0\right)+3\left(1+e^{-\alpha}\right)}{\alpha e^{\alpha}-1-e^{-\alpha}},
\end{equation}
where we used the fact that $x^{1/2-\sigma_{x,\alpha}}=e^{-\alpha}$. Because
\[
\Im\left\{\frac{1}{s-\rho}\right\} = \frac{\gamma-t}{\left(\sigma-\frac{1}{2}\right)^2+\left(t-\gamma\right)^2},
\]
Lemma~\ref{lem:logder} gives
\begin{multline*}
\left|\Im\left\{\frac{\zeta'}{\zeta}\left(s_{x,\alpha}\right)-\frac{\zeta'}{\zeta}\left(s\right)\right\}\right| \leq \sum_{\gamma} \left(\frac{\left(\sigma_{x,\alpha}-\frac{1}{2}\right)^2}{\left(\sigma_{x,\alpha}-\frac{1}{2}\right)^2+\left(t-\gamma\right)^2}\right. \times \\
\times\left.
\frac{\left|t-\gamma\right|}{\left(\sigma-\frac{1}{2}\right)^2+\left(t-\gamma\right)^2}\right) + \frac{\alpha}{t_0\log{x}}\left(\frac{\pi}{4}+\frac{2}{t_0}\right)
\end{multline*}
for $\sigma\in\left[1/2,\sigma_{x,\alpha}\right]$ since
\[
\left|\left(\sigma-\frac{1}{2}\right)^2-\left(\sigma_{x,\alpha}-\frac{1}{2}\right)^2\right| = \left(\sigma_{x,\alpha}-\frac{1}{2}\right)^2 - \left(\sigma-\frac{1}{2}\right)^2 \leq \left(\sigma_{x,\alpha}-\frac{1}{2}\right)^2.
\]
Therefore, by the former inequality and the definition of $J_3$ we have
\begin{multline*}
\frac{\left|J_3\right|}{\pi} \leq \frac{\sigma_{x,\alpha}-\frac{1}{2}}{\pi}\sum_{\gamma}\left(\frac{\sigma_{x,\alpha}-\frac{1}{2}}{\left(\sigma_{x,\alpha}-\frac{1}{2}\right)^2+\left(t-\gamma\right)^2} \int_{\frac{1}{2}}^{\sigma_{x,\alpha}}\frac{\left|t-\gamma\right|\dif{\sigma}}{\left(\sigma-\frac{1}{2}\right)^2+\left(t-\gamma\right)^2}\right) \\
+ \frac{\alpha^2}{\pi t_0\log^{2}{x}}\left(\frac{\pi}{4}+\frac{2}{t_0}\right).
\end{multline*}
Because
\[
\int_{\frac{1}{2}}^{\sigma_{x,\alpha}}\frac{\left|t-\gamma\right|\dif{\sigma}}{\left(\sigma-\frac{1}{2}\right)^2+\left(t-\gamma\right)^2} \leq \int_{0}^{\infty} \frac{\left|t-\gamma\right|\dif{u}}{u^2+\left(t-\gamma\right)^2} \leq \frac{\pi}{2},
\]
inequality~\eqref{eq:important} guarantees that
\begin{flalign}
\frac{\left|J_3\right|}{\pi} &\leq \frac{\alpha}{2\log{x}}F\left(s_{x,\alpha}\right) + \frac{\alpha^2}{\pi t_0\log^2{x}}\left(\frac{\pi}{4}+\frac{2}{t_0}\right) \nonumber \\
&\leq \frac{\alpha^2 e^{\alpha}}{\alpha e^{\alpha}-1-e^{-\alpha}}\left(\frac{\left|r(x,t,\alpha)\right|}{2\log{x}}+\frac{\log{t}}{4\log{x}}\right) \nonumber \\
&+ \frac{\alpha^2\left(A\left(x_0\right)+3e^{\alpha}\right)}{2\left(\alpha e^{\alpha}-1-e^{-\alpha}\right)}\frac{1}{\log{x}}
+ \frac{\alpha^2}{\pi t_0}\left(\frac{\pi}{4}+\frac{2}{t_0}\right)\frac{1}{\log^2{x}}. \label{eq:J3}
\end{flalign}
Using~\eqref{eq:J2} and~\eqref{eq:J3} in~\eqref{eq:S} gives
\begin{flalign}
\label{eq:SJ1}
\left|S(t)-\frac{J_1}{\pi}\right| &\leq \frac{\alpha^2 e^{\alpha}\left(2+\pi\right)}{2\pi\left(\alpha e^{\alpha}-1-e^{-\alpha}\right)}\frac{|r(x,t,\alpha)|}{\log{x}} + \frac{\alpha^2e^{\alpha}\pi+2\alpha\left(1+e^{-\alpha}\right)}{4\pi\left(\alpha e^{\alpha}-1-e^{-\alpha}\right)}\frac{\log{t}}{\log{x}} \nonumber \\
&+\alpha\left(\frac{\alpha\left(2+\pi\right)A\left(x_0\right)+3\pi\alpha e^{\alpha}+6\left(1+e^{-\alpha}\right)}{2\pi\left(\alpha e^{\alpha}-1-e^{-\alpha}\right)}\right)\frac{1}{\log{x}} \nonumber \\
&+ \frac{\alpha^2}{\pi t_0}\left(\frac{\pi}{4}+\frac{2}{t_0}\right)\frac{1}{\log^2{x}}.
\end{flalign}
Also,
\[
\left|\frac{1}{\pi}J_1+\frac{1}{\pi}\sum_{n\leq x^2}\frac{\Lambda_x(n)\sin\left(t\log{n}\right)}{n^{\sigma_{x,\alpha}}\log{n}}\right| = \frac{1}{\pi}\left|\int_{\sigma_{x,\alpha}}^{\infty}\Im\left\{\frac{\zeta'}{\zeta}\left(s\right)+\sum_{n\leq x^2}\frac{\Lambda_x(n)}{n^{s}}\right\}\dif{\sigma}\right|,
\]
and the right-hand side of the latter equality is not greater than
\begin{equation}
\label{eq:J1sum}
\frac{1}{\pi\left(\alpha e^{\alpha}-1-e^{-\alpha}\right)}\left(\left(1+\frac{1}{2e^{\alpha}}\right)\left(\frac{|r(x,t,\alpha)|}{\log{x}}+\frac{\log{t}}{2\log{x}}\right)+\frac{\alpha A\left(x_0\right)+3\left(1+e^{-\alpha}\right)}{\log{x}}\right)
\end{equation}
by inequality~\eqref{eq:main2} since
\[
\int_{\sigma_{x,\alpha}}^{\infty}x^{\frac{1}{2}-\sigma}\dif{\sigma} = \frac{e^{-\alpha}}{\log{x}}, \quad \int_{\sigma_{x,\alpha}}^{\infty}\left(1+x^{\frac{1}{2}-\sigma}\right)x^{\frac{1}{2}-\sigma}\dif{\sigma} = \frac{\left(2+e^{-\alpha}\right)e^{-\alpha}}{2\log{x}}.
\]
Combining~\eqref{eq:SJ1} and~\eqref{eq:J1sum} gives the bound from the theorem.
\end{proof}

\subsection{Approximation of $S_1(t)$ by segments of Dirichlet series}
\label{sec:S1}

Taking a similar approach as in the proof of Theorem~\ref{thm:selberg}, we can provide a proof of the analogous result for $S_1(t)$.

\begin{theorem}
\label{thm:selberg2}
Assume the Riemann Hypothesis. Let $1\leq\alpha\leq2$, $t\geq t_0\geq10$, and $e^{4/3}\leq x_0\leq x\leq t^2$. With the notations from Lemma~\ref{lem:selberg} we then have
\begin{multline*}
\left|S_1(t)+\frac{1}{\pi}\mathrm{p.v.}\int_{\frac{1}{2}}^{\infty}\log{\left|\zeta(\sigma)\right|}\dif{\sigma}-\frac{1}{\pi}\sum_{n\leq x^2}\frac{\Lambda_x(n)\cos{\left(t\log{n}\right)}}{n^{\sigma_{x,\alpha}}\log^2{n}}\left(1+\frac{\alpha\log{n}}{\log{x}}\right)\right| \\
\leq a_2(\alpha)\frac{\left|r(x,t,\alpha)\right|}{\log^2{x}} + b_2(\alpha)\frac{\log{t}}{\log^2{x}} + \frac{c_2\left(x_0,\alpha\right)}{\log^2{x}} + \frac{d_2\left(t_0,\alpha\right)}{\log^3{x}},
\end{multline*}
where
\begin{gather*}
a_2(\alpha) \de \frac{\alpha^3 e^{\alpha} + 1 + \alpha + \frac{1+2\alpha}{4}e^{-\alpha}}{\pi\left(\alpha e^{\alpha}-1-e^{-\alpha}\right)}, \\
b_2(\alpha)\de \frac{\alpha^3 e^{\alpha} + \left(\alpha^2+\alpha+\frac{1}{2}\right)e^{-\alpha}+\alpha^2+2(1+\alpha)}{4\pi\left(\alpha e^{\alpha}-1-e^{-\alpha}\right)}, \\
c_2\left(x_0,\alpha\right) \de \frac{3\alpha^3e^{\alpha}+2\alpha\left(\alpha^2+\alpha+1\right)A\left(x_0\right)+3\left(\alpha^2+2\alpha+2\right)\left(1+e^{-\alpha}\right)}{2\pi\left(\alpha e^{\alpha}-1-e^{-\alpha}\right)}, \\
d_2\left(t_0,\alpha\right) \de \frac{\alpha^3}{2\pi t_0}\left(\frac{\pi}{4}+\frac{2}{t_0}\right),
\end{gather*}
and $A\left(x_0\right)$ is defined by~\eqref{eq:A}.
\end{theorem}

\begin{proof}
Let $s_{x,\alpha}=\sigma_{x,\alpha}+\ie t$. By~\eqref{eq:intS1} we can write
\[
\pi S_1(t)=-\mathrm{p.v.}\int_{\frac{1}{2}}^{\infty}\log{\left|\zeta(\sigma)\right|}\dif{\sigma}+J_4+J_5+J_6,
\]
where
\begin{gather*}
J_4\de -\int_{\sigma_{x,\alpha}}^{\infty}\left(\sigma-\frac{1}{2}\right)\Re\left\{\frac{\zeta'}{\zeta}\left(s\right)\right\}\dif{\sigma}, \quad
J_5\de -\frac{1}{2}\left(\sigma_{x,\alpha}-\frac{1}{2}\right)^2\Re\left\{\frac{\zeta'}{\zeta}\left(s_{x,\alpha}\right)\right\}, \\
J_6 \de \int_{\frac{1}{2}}^{\sigma_{x,\alpha}}\left(\sigma-\frac{1}{2}\right)\Re\left\{\frac{\zeta'}{\zeta}\left(s_{x,\alpha}\right)-\frac{\zeta'}{\zeta}\left(s\right)\right\}\dif{\sigma}.
\end{gather*}
This implies
\begin{equation}
\label{eq:S1}
\left|S_1(t)+\frac{1}{\pi}\mathrm{p.v.}\int_{\frac{1}{2}}^{\infty}\log{\left|\zeta(\sigma)\right|}\dif{\sigma}-\frac{J_4}{\pi}\right| \leq \frac{\alpha^2}{2\pi\log^2{x}}\left|\frac{\zeta'}{\zeta}\left(s_{x,\alpha}\right)\right| + \frac{\left|J_6\right|}{\pi}.
\end{equation}
By Lemma~\ref{lem:logder} we have
\[
\left|J_6 - \int_{\frac{1}{2}}^{\sigma_{x,\alpha}}\left(\sigma-\frac{1}{2}\right)\left(F\left(s_{x,\alpha}\right)-F(s)\right)\dif{\sigma}\right| \leq \frac{\alpha^3}{2t_0\log^3{x}}\left(\frac{\pi}{4}+\frac{2}{t_0}\right),
\]
where $F(s)$ is defined by~\eqref{eq:F}. We also have
\[
\int_{\frac{1}{2}}^{\sigma_{x,\alpha}}\left(\sigma-\frac{1}{2}\right)\left(F\left(s_{x,\alpha}\right)-F(s)\right)\dif{\sigma} = \sum_{\gamma}\frac{K(\gamma)}{\left(\sigma_{x,\alpha}-\frac{1}{2}\right)^2+\left(t-\gamma\right)^2},
\]
where
\[
K(\gamma) \de \int_{\frac{1}{2}}^{\sigma_{x,\alpha}}\frac{\left(\sigma-\frac{1}{2}\right)\left(\sigma_{x,\alpha}-\sigma\right)\left(\left(t-\gamma\right)^2-\left(\sigma-\frac{1}{2}\right)\left(\sigma_{x,\alpha}-\frac{1}{2}\right)\right)}{\left(\sigma-\frac{1}{2}\right)^2+\left(t-\gamma\right)^2}\dif{\sigma}.
\]
Fujii showed that $\left|K(\gamma)\right|\leq (1/2)\left(\sigma_{x,\alpha}-1/2\right)^3$, see the last line of~\cite[p.~144]{FujiiANote}. Observe that calculation of $\left|K(\gamma)\right|$ for $t=\gamma$ confirms that this inequality is sharp. Inequality~\eqref{eq:important} then implies
\begin{flalign}
\label{eq:J6}
\frac{\left|J_6\right|}{\pi} &\leq \frac{\alpha^2}{2\pi\log^2{x}}F\left(s_{x,\alpha}\right) + \frac{\alpha^3}{2\pi t_0\log^3{x}}\left(\frac{\pi}{4}+\frac{2}{t_0}\right) \nonumber \\
&\leq \frac{\alpha^3e^{\alpha}}{\alpha e^{\alpha}-1-e^{-\alpha}}\left(\frac{\left|r(x,t,\alpha)\right|}{2\pi\log^2{x}}+\frac{\log{t}}{4\pi\log^2{x}}\right) \nonumber \\
&+ \frac{\alpha^3\left(A\left(x_0\right)+3e^{\alpha}\right)}{2\pi\left(\alpha e^{\alpha}-1-e^{-\alpha}\right)}\frac{1}{\log^2{x}} + \frac{\alpha^3}{2\pi t_0}\left(\frac{\pi}{4}+\frac{2}{t_0}\right)\frac{1}{\log^3{x}}.
\end{flalign}
Using~\eqref{eq:J2} and~\eqref{eq:J6} in~\eqref{eq:S1} gives
\begin{multline}
\label{eq:S1J4}
\left|S_1(t)+\frac{1}{\pi}\mathrm{p.v.}\int_{\frac{1}{2}}^{\infty}\log{\left|\zeta(\sigma)\right|}\dif{\sigma}-\frac{J_4}{\pi}\right| \leq \frac{\alpha^3e^{\alpha}}{\pi\left(\alpha e^{\alpha}-1-e^{-\alpha}\right)}\frac{\left|r(x,t,\alpha)\right|}{\log^2{x}} \\
+ \frac{\alpha^3 e^{\alpha}+\alpha^2\left(1+e^{-\alpha}\right)}{4\pi\left(\alpha e^{\alpha}-1-e^{-\alpha}\right)}\frac{\log{t}}{\log^2{x}}
+ \frac{2\alpha^3 A\left(x_0\right)+3\alpha^2\left(\alpha e^{\alpha}+1+e^{-\alpha}\right)}{2\pi\left(\alpha e^{\alpha}-1-e^{-\alpha}\right)}\frac{1}{\log^2{x}} \\
+ \frac{\alpha^3}{2\pi t_0}\left(\frac{\pi}{4}+\frac{2}{t_0}\right)\frac{1}{\log^3{x}}.
\end{multline}
Also,
\begin{multline*}
\left|\frac{J_4}{\pi}-\frac{1}{\pi}\sum_{n\leq x^2}\frac{\Lambda_x(n)\cos{\left(t\log{n}\right)}}{n^{\sigma_{x,\alpha}}\log^2{n}}\left(1+\frac{\alpha\log{n}}{\log{x}}\right)\right| \\
=\frac{1}{\pi}\left|\int_{\sigma_{x,\alpha}}^{\infty}\left(\sigma-\frac{1}{2}\right)\Re\left\{\frac{\zeta'}{\zeta}(s)+\sum_{n\leq x^2}\frac{\Lambda_x(n)}{n^s}\right\}\dif{\sigma}\right|,
\end{multline*}
and the right-hand side of the latter equality is not greater than
\begin{equation}
\label{eq:J4sum}
\frac{1+\alpha+\frac{1+2\alpha}{4}e^{-\alpha}}{\pi\left(\alpha e^{\alpha}-1-e^{-\alpha}\right)}\left(\frac{|r(x,t,\alpha)|}{\log^2{x}}+\frac{\log{t}}{2\log^2{x}}\right) + \frac{\left(1+\alpha\right)\left(\alpha A\left(x_0\right)+3\left(1+e^{-\alpha}\right)\right)}{\pi\left(\alpha e^{\alpha}-1-e^{-\alpha}\right)\log^2{x}}
\end{equation}
by inequality~\eqref{eq:main2} since
\begin{gather*}
\int_{\sigma_{x,\alpha}}^{\infty} \left(\sigma-\frac{1}{2}\right)x^{\frac{1}{2}-\sigma}\dif{\sigma} = \frac{(1+\alpha)e^{-\alpha}}{\log^2{x}}, \\ \int_{\sigma_{x,\alpha}}^{\infty} \left(\sigma-\frac{1}{2}\right)\left(1+x^{\frac{1}{2}-\sigma}\right)x^{\frac{1}{2}-\sigma}\dif{\sigma} = \frac{\left(1+\alpha+\frac{1+2\alpha}{4}e^{-\alpha}\right)e^{-\alpha}}{\log^2{x}}.
\end{gather*}
Combining~\eqref{eq:S1J4} and~\eqref{eq:J4sum} gives the bound from the theorem.
\end{proof}

\subsection{Upper bound for $\zeta(1/2+\ie t)$}
\label{sec:zetaCL}

Following Soundararajan's proof of his main proposition in~\cite{SoundMoments}, we can prove the next theorem.

\begin{theorem}
\label{thm:logzeta}
Assume the Riemann Hypothesis. Let $\alpha\geq0.49123$, $t\geq10$, and $2\leq x\leq t^2$. With the notations from Lemma~\ref{lem:selberg} we then have
\begin{multline*}
\log{\left|\zeta\left(\frac{1}{2}+\ie t\right)\right|} \leq \sum_{n\leq x}\frac{\Lambda(n)\log{\left(\frac{x}{n}\right)}\cos\left(t\log{n}\right)}{n^{\sigma_{x,\alpha}}\log{n}\log{x}} + \frac{1+\alpha}{2}\frac{\log{t}}{\log{x}} \\
+ \frac{3\left(1+\alpha\right)}{\log{x}} + \frac{e^{-\alpha}}{\sqrt{x}\log^2{x}}\left(1+\frac{1}{x\left(x^2-1\right)}\right),
\end{multline*}
where $t$ does not coincide with the ordinate of a zero of $\zeta$.
\end{theorem}

\begin{proof}
Let $s_{x,\alpha}=\sigma_{x,\alpha}+\ie t$. By Lemma~\ref{lem:logder} we have
\begin{equation}
\label{eq:RealPart}
-\Re\left\{\frac{\zeta'}{\zeta}(s)\right\} \leq \frac{1}{2}\log{t} + 3 - F(s),
\end{equation}
where $F(s)$ is defined by~\eqref{eq:F}. Integrating this inequality over $\sigma$ from $1/2$ to $\sigma_{x,\alpha}$ gives
\begin{flalign}
\label{eq:boundCritical}
\log{\left|\zeta\left(\frac{1}{2}+\ie t\right)\right|} &\leq \log{\left|\zeta\left(s_{x,\alpha}\right)\right|} + \frac{\alpha}{\log{x}}\left(\frac{1}{2}\log{t} + 3\right) \nonumber \\
&- \frac{1}{2}\sum_{\gamma}\log{\left(1+\left(\frac{\sigma_{x,\alpha}-\frac{1}{2}}{t-\gamma}\right)^2\right)} \nonumber \\
&\leq \log{\left|\zeta\left(s_{x,\alpha}\right)\right|} + \frac{\alpha\log{t}}{2\log{x}} + \frac{3\alpha}{\log{x}} - \frac{\alpha F\left(s_{x,\alpha}\right)}{2\log{x}},
\end{flalign}
where we also used inequality
\begin{equation}
\label{eq:ineq}
\log{(1+u)} \geq \frac{u}{1+u},
\end{equation}
which is valid for $u\geq0$. Integrating equality~\eqref{eq:sound} over $\sigma$ from $\sigma_{x,\alpha}$ to $\infty$, and then extracting the real parts gives
\begin{equation}
\label{eq:sxalpha}
\log{\left|\zeta\left(s_{x,\alpha}\right)\right|} = \sum_{n\leq x}\frac{\Lambda(n)\log{\frac{x}{n}}\cos\left(t\log{n}\right)}{n^{\sigma_{x,\alpha}}\log{n}\log{x}} - \frac{1}{\log{x}}\Re\left\{\frac{\zeta'}{\zeta}\left(s_{x,\alpha}\right)\right\} + J_7 + J_8 + J_9,
\end{equation}
where
\begin{gather*}
J_7 \de -\frac{1}{\log{x}}\Re\left\{\int_{\sigma_{x,\alpha}}^{\infty}\frac{x^{1-s}}{(1-s)^2}\dif{\sigma}\right\}, \quad J_8 \de \frac{1}{\log{x}}\Re\left\{\int_{\sigma_{x,\alpha}}^{\infty}\sum_{q=1}^{\infty}\frac{x^{-2q-s}}{(2q+s)^2}\dif{\sigma}\right\}, \\
J_9 \de \frac{1}{\log{x}}\Re\left\{\int_{\sigma_{x,\alpha}}^{\infty}\sum_{\rho}\frac{x^{\rho-s}}{\left(\rho-s\right)^2}\dif{\sigma}\right\}.
\end{gather*}
We have
\begin{gather*}
\left|J_7\right| \leq \frac{1}{\log{x}}\int_{\sigma_{x,\alpha}}^{\infty}x^{-\sigma}\dif{\sigma} = \frac{e^{-\alpha}}{\sqrt{x}\log^2{x}}, \\
\left|J_8\right| \leq  \frac{1}{\log{x}}\sum_{q=1}^{\infty}x^{-2q-1}\int_{\sigma_{x,\alpha}}^{\infty}x^{-\sigma}\dif{\sigma} = \frac{e^{-\alpha}}{\left(x^2-1\right)x\sqrt{x}\log^2{x}}, \\
\left|J_9\right| \leq \frac{F\left(s_{x,\alpha}\right)}{\alpha}\int_{\sigma_{x,\alpha}}^{\infty}x^{\frac{1}{2}-\sigma}\dif{\sigma} = \frac{e^{-\alpha}F\left(s_{x,\alpha}\right)}{\alpha\log{x}}
\end{gather*}
since $\left|1-s\right|^2\geq t^2\geq x$, $\left|2q+s\right|^2\geq t^2\geq x$ and
\[
\left|\sum_{\rho}\frac{x^{\rho-s}}{\left(\rho-s\right)^2}\right| \leq \frac{F\left(s_{x,\alpha}\right)\log{x}}{\alpha}x^{\frac{1}{2}-\sigma}.
\]
Taking~\eqref{eq:RealPart} into~\eqref{eq:sxalpha} together with bounds on $J_7$, $J_8$ and $J_9$, and then employing~\eqref{eq:boundCritical}, we obtain the main inequality from the theorem with the additional term
\[
\frac{F\left(s_{x,\alpha}\right)}{\log{x}}\left(\frac{e^{-\alpha}}{\alpha}-\frac{\alpha}{2}-1\right).
\]
But this term is negative for $\alpha\geq 0.49123$, therefore it can be omitted. The proof is thus complete.
\end{proof}

\section{Derivation of general bounds II}
\label{sec:boundsII}

The aim of this section is to obtain explicit upper bounds for $\left|S(t)\right|$, $\left|S_1(t)\right|$, and $\left|\zeta\left(1/2+\ie t\right)\right|$ in terms of elementary functions by using results from Section~\ref{sec:boundsI}. Therefore, we need firstly obtain explicit estimates of sums which involve $\Lambda(n)$ and $\Lambda_x(n)$ from Theorems~\ref{thm:selberg},~\ref{thm:selberg2}, and~\ref{thm:logzeta}, namely
\begin{equation}
\label{eq:rbound}
\left|r(x,t,\alpha)\right| \leq \sum_{n\leq x^2}\frac{\Lambda_x(n)}{\sqrt{n}} \ll \frac{x}{\log{x}},
\end{equation}
\begin{equation}
\label{eq:sinbound}
\left|\sum_{n\leq x^2}\frac{\Lambda_x(n)\sin{\left(t\log{n}\right)}}{n^{\sigma_{x,\alpha}}\log{n}}\right| \leq \sum_{n\leq x^2}\frac{\Lambda_x(n)}{\sqrt{n}\log{n}} \ll \frac{x}{\log^2{x}},
\end{equation}
\begin{flalign}
\left|\sum_{n\leq x^2}\frac{\Lambda_x(n)\cos{\left(t\log{n}\right)}}{n^{\sigma_{x,\alpha}}\log^{2}{n}}\left(1+\frac{\alpha\log{n}}{\log{x}}\right)\right| &\leq \sum_{n\leq x^2}\frac{\Lambda_x(n)}{\sqrt{n}\log^{2}{n}} + \frac{\alpha}{\log{x}}\sum_{n\leq x^2}\frac{\Lambda_x(n)}{\sqrt{n}\log{n}} \nonumber \\
&\ll \frac{x}{\log^3{x}}, \label{eq:cosbound}
\end{flalign}
and
\begin{equation}
\label{eq:soundbound}
\left|\sum_{n\leq x}\frac{\Lambda(n)\log{\left(\frac{x}{n}\right)}\cos\left(t\log{n}\right)}{n^{\sigma_{x,\alpha}}\log{n}\log{x}}\right| \leq \frac{1}{\log{x}}\sum_{n\leq x}\frac{\Lambda(n)\log{\frac{x}{n}}}{\sqrt{n}\log{n}} \ll \frac{\sqrt{x}}{\log^2{x}}.
\end{equation}
We are already able to deduce $C_0=b_1(\alpha)$ and $C_1=b_2(\alpha)$ from these bounds after taking $x=\log{t}$ in Theorems~\ref{thm:selberg} and~\ref{thm:selberg2}, where $C_0$ and $C_1$ are constants from~\eqref{eq:SRH} and~\eqref{eq:S1RH}, respectively. With the choice $\alpha=1$ we obtain Fujii's values $C_0=0.80398\ldots$ and $C_1=0.50902\ldots$, but at $\alpha=1.5582\ldots$ and $\alpha=1.4604\ldots$ the functions $b_1(\alpha)$ and $b_2(\alpha)$ have minimal values $0.54261\ldots$ and $0.33696\ldots$, respectively.

Before proceeding to the proofs of Corollaries~\ref{cor:S},~\ref{cor:S1}, and~\ref{cor:sound}, it is useful to provide upper bounds for some sums over prime numbers, see Lemma~\ref{lem:sumsOverPrimes}. In order to do so, we state in the following two lemmas simple upper bounds for two integrals which arise from the summation formula.

\begin{lemma}
\label{lem:integral}
We have
\[
\int_{2}^{x}\frac{\dif{y}}{\sqrt{y}\log^3{y}} \leq 23.1\frac{\sqrt{x}}{\log^3{x}}
\]
for all $x\geq2$.
\end{lemma}

\begin{proof}
Let $2\leq x_0\leq x$. Then $2\leq x^{\log{2}/\log{x_0}}\leq x$, which implies
\begin{flalign*}
\int_{2}^{x}\frac{\dif{y}}{\sqrt{y}\log^3{y}} &= \left(\int_{2}^{x^{\frac{\log{2}}{\log{x_0}}}}+\int_{x^{\frac{\log{2}}{\log{x_0}}}}^{x}\right)\frac{\dif{y}}{\sqrt{y}\log^3{y}} \\
&\leq \frac{2}{\log^3{2}}\left(x^{\frac{\log{2}}{2\log{x_0}}}-\sqrt{2}\right) + \frac{2}{\left(\frac{\log{2}}{\log{x_0}}\log{x}\right)^3}\left(\sqrt{x}-x^{\frac{\log{2}}{2\log{x_0}}}\right).
\end{flalign*}
This means that
\begin{equation}
\label{eq:boundIntegral}
\frac{\log^3{x}}{\sqrt{x}}\int_{2}^{x}\frac{\dif{y}}{\sqrt{y}\log^3{y}} \leq 2\left(\frac{\log{x_0}}{\log{2}}\right)^3
+ \frac{2}{\log^3{2}}x^{-\frac{1}{2}\left(1-\frac{\log{2}}{\log{x_0}}\right)}\log^3{x}.
\end{equation}
Using \emph{Mathematica}, we expressed the function on the left-hand side of~\eqref{eq:boundIntegral} with the exponential integral function and used \texttt{FindMaximum} to verify that it is not greater than $23.1$ for $x\in\left[2,10^{26}\right]$. Alternatively, we used inequality
\[
\frac{\log^3{x}}{\sqrt{x}}\int_{2}^{x}\frac{\dif{y}}{\sqrt{y}\log^3{y}} \leq \frac{\log^3{x_1}}{\sqrt{x_1}}\int_{2}^{x_2}\frac{\dif{y}}{\sqrt{y}\log^3{y}}
\]
for $x\in\left[x_1,x_2\right]$ and $x_1\geq10^3$ to verify this for $x\in\bigcup_{0\leq n\leq 10}\left[10^{5+2n},10^{7+2n}\right]$, while for $x\in\left[2,10^5\right]$ we split this interval into small intervals and calculated values at endpoints of each interval. Taking $\log{x_0}=\sqrt[3]{23/2}\log{2}$ in~\eqref{eq:boundIntegral}, we can see that this is true also for $x\geq10^{26}$. This proves the lemma.
\end{proof}

\begin{lemma}
\label{lem:integral2}
We have
\[
\int_{2}^{x}\frac{\dif{y}}{\sqrt{y}\log^4{y}} \leq 132.6\frac{\sqrt{x}}{\log^4{x}}
\]
for all $x\geq2$.
\end{lemma}

\begin{proof}
Following the same approach as in the proof of Lemma~\ref{lem:integral}, we obtain
\begin{equation}
\label{eq:boundIntegral2}
\frac{\log^4{x}}{\sqrt{x}}\int_{2}^{x}\frac{\dif{y}}{\sqrt{y}\log^4{y}} \leq 2\left(\frac{\log{x_0}}{\log{2}}\right)^4
+ \frac{2}{\log^4{2}}x^{-\frac{1}{2}\left(1-\frac{\log{2}}{\log{x_0}}\right)}\log^4{x}
\end{equation}
for $2\leq x_0\leq x$. Similarly as before we can show that the function on the left-hand side of~\eqref{eq:boundIntegral2} is not greater than $132.6$ for $x\in\left[2,10^{26}\right]$. Taking $\log{x_0}=\sqrt[4]{132/2}\log{2}$ in~\eqref{eq:boundIntegral2}, we can see that this is true also for $x\geq 10^{26}$, which concludes the proof.
\end{proof}

\begin{lemma}
\label{lem:sumsOverPrimes}
Assume the Riemann Hypothesis. Let $X\geq2$. Then
\[
\sum_{p\leq X} \frac{\log{p}}{\sqrt{p}} \leq 2\sqrt{X} + \frac{\log^3{X}}{48\pi} + \frac{\log^2{X}}{8\pi} - 1.41,
\]
\[
\sum_{p\leq X}\frac{\log{p}\log{\frac{X}{p}}}{\sqrt{p}} \leq 4\sqrt{X} + \frac{\log^4{X}}{192\pi} +\frac{\log^3{X}}{24\pi} - 1.416\log{X} - 4.679,
\]
\[
\sum_{p\leq X}\frac{1}{\sqrt{p}} \leq \frac{2\sqrt{X}}{\log{X}} + \frac{4\sqrt{X}}{\log^2{X}} + \frac{184.8\sqrt{X}}{\log^3{X}} + \frac{\log^2{X}}{32\pi} + \frac{\log{X}}{4\pi} - 13.84,
\]
\[
\sum_{p\leq X}\frac{1}{\sqrt{p}\log{p}} \leq \frac{2\sqrt{X}}{\log^2{X}} + \frac{92.4\sqrt{X}}{\log^3{X}} + \frac{\log{X}}{16\pi} + \frac{\log{\log{X}}}{4\pi} - 2.888,
\]
\[
\sum_{p\leq X}\frac{\log{\frac{X}{p}}}{\sqrt{p}} \leq \frac{4\sqrt{X}}{\log{X}} + \frac{184.8\sqrt{X}}{\log^2{X}} + \frac{\log^3{X}}{96\pi} + \frac{\log^2{X}}{8\pi} - 13.84\log{X} + 1.417,
\]
and
\begin{flalign*}
\sum_{p\leq X}\frac{\log{\frac{X}{p}}}{\sqrt{p}\log{p}} \leq \frac{4\sqrt{X}}{\log^2{X}} + \frac{3166.4\sqrt{X}}{\log^3{X}} &+ \frac{\log^2{X}}{32\pi} \\
&+ \frac{\log{X}\log{\log{X}}}{4\pi} - 36.94\log{X} + 81.8.
\end{flalign*}
\end{lemma}

\begin{proof}
We are using
\begin{equation}
\label{eq:SummationFormula}
\sum_{p\leq X} f(p) = \int_{2}^{X} \frac{f(y)}{\log{y}}\dif{y} + \frac{2f(2)}{\log{2}} + \frac{f(X)\left(\vartheta(X)-X\right)}{\log{X}}-\int_{2}^{X} \left(\vartheta(y)-y\right)\left(\frac{f(y)}{\log{y}}\right)'\dif{y}
\end{equation}
for a differentiable function $f(y)$, see~\cite[Equation 4.14]{RosserSchoenfeld}, together with
\[
\vartheta(X)-X\leq\frac{1}{8\pi}\sqrt{X}\log^2{X},
\]
which is valid under the Riemann Hypothesis for all $X\geq 2$, see~\cite[Equation 6.5]{SchoenfeldSharperRH}. As usual, $\vartheta(X)=\sum_{p\leq X}\log{p}$ is one of Chebyshev's functions. In our application of \eqref{eq:SummationFormula} we have
\[
\frac{f(y)}{\log{y}} = \frac{\log^{\mu_1}{\left(X/y\right)}}{\sqrt{y}\log^{\mu_2}{y}}
\]
for $\mu_1\in\{0,1\}$ and $\mu_2\in\{0,1,2\}$ since pairs $\left(\mu_1,\mu_2\right)=(0,0)$, $\left(\mu_1,\mu_2\right)=(1,0)$, $\left(\mu_1,\mu_2\right)=(0,1)$, $\left(\mu_1,\mu_2\right)=(0,2)$, $\left(\mu_1,\mu_2\right)=(1,1)$, and $\left(\mu_1,\mu_2\right)=(1,2)$ correspond in the same order to the sums from Lemma \ref{lem:sumsOverPrimes}. It is not hard to see that
\begin{multline*}
\sum_{p\leq X} f(p) \leq \int_{2}^{X}\frac{\dif{y}}{\sqrt{y}\log^{\mu_2}{y}} + \frac{1}{8\pi}\log^{2-\mu_2}{X} \\
+\frac{\mu_2}{8\pi}\int_{2}^{X}\frac{\log^{1-\mu_2}{y}}{y}\dif{y}
+\frac{1}{16\pi}\int_{2}^{X}\frac{\log^{2-\mu_2}{y}}{y}\dif{y}
+ \frac{\sqrt{2}}{\log^{\mu_2}{2}}
\end{multline*}
if $\mu_1=0$, while for $\mu_1=1$ we have
\begin{flalign*}
\sum_{p\leq X} f(p) &\leq \log{X}\int_{2}^{X}\frac{\dif{y}}{\sqrt{y}\log^{\mu_2}{y}} - \int_{2}^{X}\frac{\dif{y}}{\sqrt{y}\log^{\mu_2-1}{y}} + \frac{\mu_2\log{X}}{8\pi}\int_{2}^{X}\frac{\log^{1-\mu_2}{y}}{y}\dif{y} \\
&+\frac{1}{8\pi}\left(\frac{\log{X}}{2}+1-\mu_2\right)\int_{2}^{X}\frac{\log^{2-\mu_2}{y}}{y}\dif{y}-\frac{1}{16\pi}\int_{2}^{X}\frac{\log^{3-\mu_2}{y}}{y}\dif{y} \\
&+ \frac{\sqrt{2}}{\log^{\mu_2}{2}}\log{X} - \frac{\sqrt{2}}{\log^{\mu_2-1}{2}}.
\end{flalign*}
It remains to show how to estimate the above integrals. Integration by parts implies
\[
\int_{2}^{X}\frac{\log^{n}{y}}{y}\dif{y} = \frac{\log^{n+1}{X}-\log^{n+1}{2}}{n+1}
\]
for $n\geq0$, and also
\begin{gather}
\int_{2}^{X}\frac{\dif{y}}{\sqrt{y}\log{y}} = \frac{2\sqrt{X}}{\log{X}} + \frac{4\sqrt{X}}{\log^2{X}} - \frac{2\sqrt{2}}{\log{2}} - \frac{4\sqrt{2}}{\log^2{2}} + 8\int_{2}^{X}\frac{\dif{y}}{\sqrt{y}\log^3{y}}, \label{eq:integral1} \\
\int_{2}^{X}\frac{\dif{y}}{\sqrt{y}\log^{2}{y}} = \frac{2\sqrt{X}}{\log^2{X}} - \frac{2\sqrt{2}}{\log^2{2}} + 4\int_{2}^{X}\frac{\dif{y}}{\sqrt{y}\log^3{y}}, \label{eq:integral2} \\
\int_{2}^{X}\frac{\dif{y}}{\sqrt{y}\log^{3}{y}} = \frac{2\sqrt{X}}{\log^3{X}} - \frac{2\sqrt{2}}{\log^3{2}} + 6\int_{2}^{X}\frac{\dif{y}}{\sqrt{y}\log^4{y}}. \label{eq:integral3}
\end{gather}
The stated upper bounds from Lemma \ref{lem:sumsOverPrimes} now follow by tedious but straightforward calculations. We used Lemma~\ref{lem:integral} in the estimation of \eqref{eq:integral1} and \eqref{eq:integral2}, while we used Lemma~\ref{lem:integral2} for \eqref{eq:integral3} which is then used only in the last inequality of the lemma.
\end{proof}

Note that the above proof also shows that the inequalities from Lemma~\ref{lem:sumsOverPrimes} are asymptotically correct, regardless of any hypothesis.

We are ready to prove the next corollaries to theorems from Section~\ref{sec:boundsI}.

\begin{corollary}
\label{cor:S}
Assume the Riemann Hypothesis. Let $1\leq\alpha\leq2$, $t\geq t_0\geq10$, and $e^{4/3}\leq x_0\leq x\leq t^2$. Then we have
\[
\left|S(t)\right| \leq b_1(\alpha)\frac{\log{t}}{\log{x}} + \mathcal{E}_1\left(x_0,t_0,\alpha;x\right)\frac{x}{\log^2{x}},
\]
where
\begin{flalign*}
\mathcal{E}_1\left(x_0,t_0,\alpha;x\right) \de 4a_1(\alpha)+\frac{2}{\pi} &+ \frac{46.2}{\pi\log{x}} + \left(2a_1(\alpha)+\frac{2}{\pi}\right)\frac{\log{x}}{\sqrt{x}} \\
&+ \frac{4}{\pi}\left(1+\frac{46.2}{\log{x}}\right)\frac{1}{\sqrt{x}} + \frac{1}{48\pi}\left(5a_1(\alpha)+\frac{11}{2\pi}\right)\frac{\log^4{x}}{x} \\
&+ \frac{1}{24\pi}\left(11a_1(\alpha)+\frac{18}{\pi}\right)\frac{\log^3{x}}{x} + \frac{\log^2{x}\log{\log{x}}}{2\pi x} \\
&+ \left(a_1(\alpha) - 12.949\right)\frac{\log^2{x}}{x} \\
&+ \left(\frac{1.417}{\pi}-1.762 a_1(\alpha)+c_1\left(x_0,\alpha\right)\right)\frac{\log{x}}{x} \\
&+ \left(\frac{1}{2\pi}-4.679a_1(\alpha)+d_1\left(t_0,\alpha\right)\right)\frac{1}{x},
\end{flalign*}
and $a_1(\alpha)$, $b_1(\alpha)$, $c_1\left(x_0,\alpha\right)$, and $d_1\left(t_0,\alpha\right)$ are functions from Theorem~\ref{thm:selberg}.
\end{corollary}

\begin{proof}
By~\eqref{eq:rbound},~\eqref{eq:sinbound}, and Theorem~\ref{thm:selberg} we have
\begin{multline}
\label{eq:ineqS}
\left|S(t)\right| \leq \frac{1}{\pi}\sum_{n\leq x^2}\frac{\Lambda_x(n)}{\sqrt{n}\log{n}} + \frac{a_1(\alpha)}{\log{x}}\sum_{n\leq x^2}\frac{\Lambda_x(n)}{\sqrt{n}} \\
+ b_1(\alpha)\frac{\log{t}}{\log{x}} + \frac{c_1\left(x_0,\alpha\right)}{\log{x}} + \frac{d_1(t_0,\alpha)}{\log^2{x}}.
\end{multline}
We are using
\begin{equation}
\label{eq:sumprimes2}
 \sum_{n\leq x^2}\frac{\Lambda_x(n)}{\sqrt{n}\log{n}} \leq \sum_{p\leq x}\frac{1}{\sqrt{p}} + \frac{1}{\log{x}}\sum_{p\leq x^2}\frac{\log{\frac{x^2}{p}}}{\sqrt{p}} + \frac{1}{2}\sum_{p\leq x}\frac{1}{p} + \sum_{r=3}^{\infty}\sum_{p} \frac{1}{rp^{r/2}}
\end{equation}
and
\begin{equation}
\label{eq:sumprimes1}
\sum_{n\leq x^2}\frac{\Lambda_x(n)}{\sqrt{n}} \leq \sum_{p\leq x}\frac{\log{p}}{\sqrt{p}} + \frac{1}{\log{x}}\sum_{p\leq x^2}\frac{\log{p}\log{\frac{x^2}{p}}}{\sqrt{p}} + \sum_{p\leq x}\frac{\log{p}}{p} + \sum_{r=3}^{\infty}\sum_{p}\frac{\log{p}}{p^{r/2}}
\end{equation}
for estimation of the first and the second sum in~\eqref{eq:ineqS}. By \eqref{eq:SummationFormula} and the unconditional bound $\left|\vartheta(X)-X\right|\leq 3X/\left(2\log{X}\right)$ for $X\geq 2$, see \cite[Theorem 4]{RosserSchoenfeld}, we have
\begin{flalign*}
\sum_{r=3}^{\infty}\sum_{p}\frac{\log{p}}{p^{r/2}} &= \left(\sum_{p\leq X_1}+\sum_{p>X_1}\right)\frac{\log{p}}{p\left(\sqrt{p}-1\right)} \\
&\leq \sum_{p\leq X_1}\frac{\log{p}}{p\left(\sqrt{p}-1\right)} + \frac{2}{\sqrt{X_1}-1}\left(1+\frac{5}{\log{X_1}}\right) \leq 2.48
\end{flalign*}
for $X_1=1.3\cdot10^{6}$. Similarly, we also have
\begin{equation}
\label{eq:doublesum}
\sum_{r=3}^{\infty}\sum_{p}\frac{1}{rp^{r/2}} \leq \sum_{p\leq X_2}\frac{1/3}{p\left(\sqrt{p}-1\right)} + \frac{2/3}{\left(\sqrt{X_2}-1\right)\log{X_2}}\left(1+\frac{7}{\log{X_2}}\right) \leq \frac{2.12}{3}
\end{equation}
for $X_2=2\cdot10^3$. By \cite[Equations 2.10, 3.20 and 3.24]{RosserSchoenfeld} it is known that
\begin{gather}
\sum_{p\leq x}\frac{\log{p}}{p} \leq \log{x}, \nonumber \\
\sum_{p\leq x}\frac{1}{p} \leq \log{\log{x}} + 0.262 + \frac{1}{\log^2{x}}. \label{eq:oneoverp}
\end{gather}
Lemma~\ref{lem:sumsOverPrimes} now easily implies the final bound.
\end{proof}

Although there exist sharper unconditional estimates for the last sums in the above proof, there is no need to include them here.

\begin{corollary}
\label{cor:S1}
Assume the Riemann Hypothesis. Let $1\leq\alpha\leq2$, $t\geq t_0\geq10$, and $e^{4/3}\leq x_0\leq x\leq t^2$. Then we have
\[
\left|S_1(t)\right| \leq b_2(\alpha)\frac{\log{t}}{\log^{2}{x}} + \mathcal{E}_2\left(x_0,t_0,\alpha;x\right)\frac{x}{\log^{3}{x}},
\]
where
\begin{flalign*}
\mathcal{E}_2\left(x_0,t_0,\alpha;x\right) \de 4a_2(\alpha)&+\frac{1+2\alpha}{\pi} + \frac{395.8+46.2\alpha}{\pi\log{x}} + 2\left(a_2(\alpha)+\frac{1+\alpha}{\pi}\right)\frac{\log{x}}{\sqrt{x}} \\
&+ \frac{92.4+4\alpha}{\pi\sqrt{x}} + \frac{184.8\alpha}{\pi\sqrt{x}\log{x}} \\
&+\left(\frac{1}{16\pi^2}\left(3+\frac{11\alpha}{6}\right)+\frac{5a_2(\alpha)}{48\pi}\right)\frac{\log^{4}{x}}{x} \\
&+ \frac{3\log^3{x}\log{\log{x}}}{4\pi^2 x} +\left(\frac{3\alpha}{4\pi^2} + \frac{11a_2(\alpha)}{24\pi} - 23.34\right)\frac{\log^3{x}}{x} \\
&+ \frac{\alpha\log^{2}{x}\log{\log{x}}}{2\pi x} + \left(\frac{81.8-40.68\alpha}{\pi}+a_2(\alpha)\right)\frac{\log^{2}{x}}{x} \\
&+ \left(\frac{1.417\alpha}{\pi}-1.762a_2(\alpha) + c_2\left(x_0,\alpha\right)\right)\frac{\log{x}}{x} \\
&+ \left(\frac{\alpha}{2\pi}-4.679a_2(\alpha)+d_2\left(t_0,\alpha\right)\right)\frac{1}{x},
\end{flalign*}
and $a_2(\alpha)$, $b_2(\alpha)$, $c_2\left(x_0,\alpha\right)$, and $d_2\left(t_0,\alpha\right)$ are functions from Theorem~\ref{thm:selberg2}.
\end{corollary}

\begin{proof}
By~\eqref{eq:rbound},~\eqref{eq:cosbound}, and Theorem~\ref{thm:selberg2} we have
\begin{flalign*}
\left|S_1(t)\right| \leq \frac{1}{\pi}\sum_{n\leq x^2}\frac{\Lambda_x(n)}{\sqrt{n}\log^{2}{n}} &+ \frac{\alpha}{\pi\log{x}}\sum_{n\leq x^2}\frac{\Lambda_x(n)}{\sqrt{n}\log{n}} + \frac{a_2(\alpha)}{\log^{2}{x}}\sum_{n\leq x^2}\frac{\Lambda_x(n)}{\sqrt{n}} \\
&+ b_2(\alpha)\frac{\log{t}}{\log^{2}{x}} + 0.82 + \frac{c_2\left(x_0,\alpha\right)}{\log^{2}{x}} + \frac{d_2\left(t_0,\alpha\right)}{\log^{3}{x}}.
\end{flalign*}
We are using inequalities~\eqref{eq:sumprimes2} and~\eqref{eq:sumprimes1} for the estimation of the second and the third sums. The first sum is bounded as
\begin{flalign*}
\sum_{n\leq x^2}\frac{\Lambda_x(n)}{\sqrt{n}\log^{2}{n}} \leq \sum_{p\leq x}\frac{1}{\sqrt{p}\log{p}} &+ \frac{1}{\log{x}}\sum_{p\leq x^2}\frac{\log{\frac{x^2}{p}}}{\sqrt{p}\log{p}} \\
&+ \frac{1}{4}\sum_{p}\frac{1}{p\log{p}} + \sum_{r=3}^{\infty}\sum_{p}\frac{1}{r^2 p^{r/2}\log{p}}.
\end{flalign*}
We also have
\[
\sum_{p}\frac{1}{p\log{p}} \leq 1.64, \quad \sum_{r=3}^{\infty}\sum_{p}\frac{1}{r^2 p^{r/2}\log{p}} \leq \frac{2.12}{9\log{2}},
\]
where the first inequality follows from~\cite[p.~6]{CohenComp} and the second inequality from~\eqref{eq:doublesum}. By Lemma~\ref{lem:sumsOverPrimes} the final bound easily follows.
\end{proof}

\begin{corollary}
\label{cor:sound}
Assume the Riemann Hypothesis. Let $\alpha\geq0.49123$, $t\geq10$, and $2\leq x\leq t^2$. Then we have
\[
\left|\zeta\left(\frac{1}{2}+\ie t\right)\right| \leq \exp{\left(\frac{1+\alpha}{2}\frac{\log{t}}{\log{x}} + \mathcal{E}_3\left(\alpha;x\right)\frac{\sqrt{x}}{\log^2{x}}\right)},
\]
where
\begin{flalign*}
\mathcal{E}_3\left(\alpha;x\right) \de 4 &+ \frac{184.8}{\log{x}} + \frac{\log^4{x}}{96\pi\sqrt{x}} + \frac{\log^3{x}}{8\pi\sqrt{x}} + \frac{\log^2{x}\log{\log{x}}}{2\sqrt{x}} - \frac{13.3\log^2{x}}{\sqrt{x}} \\
&+ \frac{\left(4.417+3\alpha\right)\log{x}}{\sqrt{x}}
+ \left(2+\frac{e^{-\alpha}}{\sqrt{x}}\left(1+\frac{1}{x\left(x^2-1\right)}\right)\right)\frac{1}{\sqrt{x}}.
\end{flalign*}
\end{corollary}

\begin{proof}
We may suppose that $t$ does not coincide with the ordinate of a nontrivial zero. By inequality~\eqref{eq:soundbound} and Theorem~\ref{thm:logzeta} we have
\begin{multline*}
\log{\left|\zeta\left(\frac{1}{2}+\ie t\right)\right|} \leq \frac{1}{\log{x}}\sum_{p\leq x}\frac{\log{\frac{x}{p}}}{\sqrt{p}} + \frac{1}{2}\sum_{p\leq\sqrt{x}}\frac{1}{p} + \sum_{r=3}^{\infty}\sum_{p} \frac{1}{rp^{r/2}} \\
+ \frac{1+\alpha}{2}\frac{\log{t}}{\log{x}} + \frac{3\left(1+\alpha\right)}{\log{x}} + \frac{e^{-\alpha}}{\sqrt{x}\log^2{x}}\left(1+\frac{1}{x\left(x^2-1\right)}\right).
\end{multline*}
By Lemma~\ref{lem:sumsOverPrimes}, together with inequalities~\eqref{eq:oneoverp} and~\eqref{eq:doublesum}, the main bound now easily follows.
\end{proof}

\section{Proof of Theorem~\ref{thm:main} and Corollary~\ref{cor:gaps}}
\label{sec:proof}

In this final section we prove Theorem~\ref{thm:main} and Corollary~\ref{cor:gaps} by using results from the previous section. All numerical computations have been made with \emph{Mathematica}.

\begin{proof}[Proof of Theorem \ref{thm:main}]
Let $t\geq45$. We are going to set $x_0=x=\log^{\lambda_i}{t}$ in the corollaries from the previous section for some specific positive real numbers $\lambda_i$ where $i\in\{1,2,3\}$. Corollaries~\ref{cor:S} and~\ref{cor:S1} then imply
\begin{gather}
\left|S(t)\right| \leq \mathcal{M}\left(\frac{b_1\left(\alpha_1\right)}{\lambda_1},\frac{\mathcal{E}_1\left(\log^{\lambda_1}{t},t,\alpha_1;\log^{\lambda_1}{t}\right)}{\lambda_1^2},1-\lambda_1;t\right)\frac{\log{t}}{\log{\log{t}}},
\label{eq:Sfinal} \\
\left|S_1(t)\right| \leq \mathcal{M}\left(\frac{b_2\left(\alpha_2\right)}{\lambda_2^2},\frac{\mathcal{E}_2\left(\log^{\lambda_2}{t},t,\alpha_2;\log^{\lambda_2}{t}\right)}{\lambda_2^3},1-\lambda_2;t\right)\frac{\log{t}}{\left(\log{\log{t}}\right)^2}, \label{eq:S1final}
\end{gather}
for $\alpha_1,\alpha_2\in[1,2]$ and $\lambda_1,\lambda_2\in\left[4/\left(3\log{\log{t}}\right),1\right)$, while Corollary~\ref{cor:sound} asserts
\begin{equation}
\label{eq:CLfinal}
\left|\zeta\left(\frac{1}{2}+\ie t\right)\right| \leq \exp{\left(\mathcal{M}\left(\frac{1+\alpha_3}{2\lambda_3},\frac{\mathcal{E}_3\left(\alpha_3;\log^{\lambda_3}{t}\right)}{\lambda_3^2},1-\frac{\lambda_3}{2};t\right)
\frac{\log{t}}{\log{\log{t}}}\right)}
\end{equation}
for $\alpha_3\geq0.49123$ and $\lambda_3\in\left[\log{2}/\log{\log{t}},2\right)$, where $\mathcal{M}$ is defined by \eqref{eq:M}.

Let $t=10^y$. We are comparing ~\eqref{eq:Sfinal} with~\eqref{eq:boundS},~\eqref{eq:S1final} with~\eqref{eq:boundS1}, and~\eqref{eq:CLfinal} with~\eqref{eq:hiary} in a way that we numerically solve appropriate equations in $y$ for admissible parameters $\alpha_i$ and $\lambda_i$, and then search for a minimal $y$. After rounding all values, we obtain $\alpha_1=1.5281$, $\lambda_1=0.715$ and $y=2465$; $\alpha_2=1.3045$, $\lambda_2=0.7295$ and $y=207.7$; $\alpha_3=0.49123$, $\lambda_3=1.4947$ and $y=39.2$. In the following three paragraphs we are going to explain how to use these values to obtain estimates from Theorem~\ref{thm:main}.

Let $t\geq 10^{2465}$, $\alpha_1=1.5281$ and $\lambda_1=0.715$. Because $c_1\left(u,\alpha_1\right)$ and $d_1\left(u,\alpha_1\right)$ are decreasing functions in the variable $u$, we can write
\[
\mathcal{E}_1\left(\log^{\lambda_1}{t},t,\alpha_1;x\right) \leq \widehat{\mathcal{E}}_1(x) \de \mathcal{E}_1\left(\left(\log{10^{2465}}\right)^{0.715},10^{2465},1.5281;x\right).
\]
By the definition of $\mathcal{E}_1$ we also have
\begin{flalign*}
\widehat{\mathcal{E}}_1(x) &= \frac{46.2}{\pi\log{x}} + \left(2a_1(1.5281)+\frac{2}{\pi}\right)\frac{\log{x}}{\sqrt{x}} - \left(12.949-a_1(1.5281)\right)\frac{\log^2{x}}{x} \\
&- \left(4.679a_1(1.5281)-\frac{1}{2\pi}-d_1\left(10^{2465},1.5281\right)\right)\frac{1}{x} + \mathcal{E}_{11}(x),
\end{flalign*}
where $\mathcal{E}_{11}(x)$ consists of other (positive) terms from $\mathcal{E}_1$. By simple analysis with derivatives we can see that $\mathcal{E}_{11}(x)$ is decreasing for $x\geq e^4$, while the function which consists of the first four terms on the right-hand side of the above equality is decreasing for $x\geq u\geq e^6$ if
\[
\frac{14.7059\sqrt{u}}{\log^3{u}} + 3.7809\left(\frac{1}{2}-\frac{1}{\log{u}}\right) > \left(11.377+\frac{7.2}{\log^2{u}}\right)\frac{\log{u}}{\sqrt{u}}.
\]
Because the latter inequality is valid for $u=10^3$, we deduce that $\widehat{\mathcal{E}}_1(x)$ decreases for $x\geq 10^3$. We verify by computer that $\widehat{\mathcal{E}}_1(x)$ also decreases for $e\leq x\leq 10^3$. Therefore,
\[
\frac{1}{\lambda_1^2}\mathcal{E}_1\left(\log^{\lambda_1}{t},t,\alpha_1;\log^{\lambda_1}{t}\right) \leq \frac{1}{\lambda_1^2}\widehat{\mathcal{E}}_1\left(\left(\log{10^{2465}}\right)^{0.715}\right) \leq 20.1911.
\]
The bound on $\left|S(t)\right|$ from Theorem~\ref{thm:main} now easily follows.

Let $t\geq 10^{208}$, $\alpha_2=1.3045$ and $\lambda_2=0.7295$. Our approach here is similar to the previous paragraph. Because $c_2\left(u,\alpha_2\right)$ and $d_2\left(u,\alpha_2\right)$ are decreasing functions in the variable $u$, we can write
\[
\mathcal{E}_2\left(\log^{\lambda_2}{t},t,\alpha_2;x\right) \leq \widehat{\mathcal{E}}_2(x) \de \mathcal{E}_2\left(\left(\log{10^{208}}\right)^{0.7295},10^{208},1.3045;x\right),
\]
where
\begin{flalign*}
\widehat{\mathcal{E}}_2(x) &= \frac{456.0679}{\pi\log{x}} + 2\left(a_2(1.3045)+\frac{2.3045}{\pi}\right)\frac{\log{x}}{\sqrt{x}} \\
&-\left(23.34-\frac{3.9135}{4\pi^2}-\frac{11a_2(1.3045)}{24\pi}\right)\frac{\log^3{x}}{x} \\
&-\left(4.679a_2(1.3045)-\frac{1.3045}{2\pi}-d_2\left(10^{208},1.3045\right)\right)\frac{1}{x} + \mathcal{E}_{21}(x).
\end{flalign*}
We can see that $\mathcal{E}_{21}(x)$ decreases for $x\geq e^4$, while the function which consists of the first four terms on the right-hand side of the above equality decreases for $x\geq u\geq e^6$ if
\[
\frac{145.17\sqrt{u}}{\log^3{u}} + 3.398\left(\frac{1}{2}-\frac{1}{\log{u}}\right) > \left(23.1+\frac{4.3114}{\log^3{u}}\right)\frac{\log^2{u}}{\sqrt{u}}.
\]
Because the latter inequality is valid for $u=10^4$, we deduce that $\widehat{\mathcal{E}}_2(x)$ decreases for $x\geq 10^4$. We verify by computer that $\widehat{\mathcal{E}}_2(x)$ also decreases for $e\leq x\leq 10^4$. Therefore,
\[
\frac{1}{\lambda_2^3}\mathcal{E}_2\left(\log^{\lambda_2}{t},t,\alpha_2;\log^{\lambda_2}{t}\right) \leq \frac{1}{\lambda_2^3}\widehat{\mathcal{E}}_2\left(\left(\log{10^{208}}\right)^{0.7295}\right) \leq 60.12.
\]
The bound on $\left|S_1(t)\right|$ from Theorem~\ref{thm:main} now easily follows.

Let $t\geq 10^{40}$, $\alpha_3=0.49123$ and $\lambda_3=1.4947$. Derivative analysis shows that
\[
\frac{184.8}{\log{x}} + \frac{\log^4{x}}{96\pi\sqrt{x}} - \frac{13.3\log^2{x}}{\sqrt{x}}
\]
is a decreasing function for $x\geq u\geq e^{10}$ if
\[
\frac{184.8\sqrt{u}}{\log^5{u}} + \frac{\log{u}}{192\pi} > \frac{1}{24\pi} + \frac{13.3}{2\log{u}}.
\]
Because the latter inequality is valid for $u=3.1\cdot10^6$, by the definition of $\widehat{\mathcal{E}_3}(x)\de\mathcal{E}_3\left(\alpha_3;x\right)$ we have $\widehat{\mathcal{E}_3}(x)\leq \widehat{\mathcal{E}_3}\left(3.1\cdot10^6\right)\leq 15.1$ for $x\geq3.1\cdot10^6$. However, as numerical calculations on the interval $\left[2,3.1\cdot10^6\right]$ reveal, $\widehat{\mathcal{E}_3}(x)$ is not always decreasing but has exactly two stationary points, local minimum $\left(1436.420\ldots,14.087\ldots\right)$ and local maximum $\left(181946.074\ldots,15.618\ldots\right)$. Therefore,
\[
\frac{1}{\lambda_3^2}\mathcal{E}_3\left(\alpha_3;\log^{\lambda_3}{t}\right) \leq \frac{1}{\lambda_3^2}\max\left\{\widehat{\mathcal{E}_3}\left(\left(\log{10^{40}}\right)^{1.4947}\right),15.62\right\} \leq 6.992.
\]
The bound on $\left|\zeta\left(1/2+\ie t\right)\right|$ from Theorem~\ref{thm:main} now easily follows. The proof of Theorem~\ref{thm:main} is thus complete.
\end{proof}

\begin{lemma}
\label{lem:gaps}
Assume that $|S(T)|\leq c_0\log{T}/\log{\log{T}}$ for some positive constant $c_0$ and $T\geq T_0>2\pi$, where also $T_0\log{\left(T_0/(2\pi)\right)}\log{\log{T_0}}>2\pi c_0$. Then there exists a nontrivial zero $\rho=\beta+\ie\gamma$ with
\[
\gamma\in\left[T,T+\frac{\widehat{c}}{\log{\log{T}}}\right]
\]
for
\begin{equation}
\label{eq:C1}
\widehat{c} > 4\pi c_0\left(1+\frac{\log{\log{T_0}}}{150c_0T_0\log{T_0}}\right)\left(1-\frac{\log{(2\pi)}}{\log{T_0}}-\frac{2\pi c_0}{T_0\log{T_0}\log{\log{T_0}}}\right)^{-1}.
\end{equation}
\end{lemma}

\begin{proof}
By the Riemann--von Mangoldt formula~\eqref{eq:RvM} with $|R(T)|\leq \frac{1}{150T}$, see \cite[Lemma 2]{BrentPlattTrudgian}, we obtain
\begin{flalign*}
N(T+H)-N(T) &= \frac{H}{2\pi}\log{\frac{T}{2\pi e}} + S(T+H) - S(T) \\
&+ \frac{T+H}{2\pi}\log{\left(1+\frac{H}{T}\right)} + R(T+H) - R(T) \\
&\geq \frac{H}{2\pi}\log{\frac{T}{2\pi}} - 2c_0\frac{\log{T}}{\log{\log{T}}} - c_0\frac{H}{T\log{\log{T}}} - \frac{2}{150T}
\end{flalign*}
for $H\geq0$, where we also applied inequality \eqref{eq:ineq} to the fourth term on the right-hand side of the above equality. Observe that $\widehat{c}>0$. Taking $H=\widehat{c}/\log{\log{T}}$, we can easily see that $N(T+H)-N(T)>0$ if $\widehat{c}$ satisfies inequality~\eqref{eq:C1}. This proves the lemma.
\end{proof}

\begin{proof}[Proof of Corollary~\ref{cor:gaps}]
Let $T\geq T_0=10^{2465}$, function $\mathcal{M}$ defined by \eqref{eq:M}, and $c_0(T)=\mathcal{M}\left(0.759282,20.1911,0.285;T\right)$. Then
\[
0.759282 < c_0(T) \leq \mathcal{M}\left(0.759282,20.1911,0.285;T_0\right).
\]
Let $\gamma\geq T_0$ be the ordinate of $\rho$. Using inequality~\eqref{eq:C1} from Lemma~\ref{lem:gaps}, where the first $c_0$ is taken as $c_0(\gamma)$ and the other two are estimated by the above inequality, we obtain the estimate from Corollary~\ref{cor:gaps}.
\end{proof}

\subsection*{Acknowledgements} The author thanks Richard Brent and David Platt for useful remarks concerning inequality~\eqref{eq:boundS1}, as well as to Micah Milinovich and Ghaith Hiary. Finally, the author is grateful to the anonymous referee for providing many valuable suggestions and corrections, and to his supervisor Tim Trudgian for continual guidance and support while writing this manuscript.


\providecommand{\bysame}{\leavevmode\hbox to3em{\hrulefill}\thinspace}
\providecommand{\MR}{\relax\ifhmode\unskip\space\fi MR }
\providecommand{\MRhref}[2]{%
  \href{http://www.ams.org/mathscinet-getitem?mr=#1}{#2}
}
\providecommand{\href}[2]{#2}

\end{document}